\documentclass[12pt,reqno]{amsart}
\usepackage{graphicx, cite, hyperref}
\usepackage{amsfonts}
\usepackage{enumerate}
\usepackage[caption = false]{subfig}
\usepackage{varioref}

\numberwithin{equation}{section}

 \theoremstyle{plain}

\newtheorem{thm}{Theorem}[section]

\newtheorem{lem}[thm]{Lemma}

\theoremstyle{definition}

\theoremstyle{remark}

\makeatother

\begin{document}

\title[Sharp Estimates on Coefficient functionals 
of Ozaki close-to-convex functions]{ Sharp Estimates on Coefficient functionals 
of Ozaki close-to-convex functions}

\author[Sushil Kumar]{Sushil Kumar}
\address{Bharati Vidyapeeth's College of Engineering,  Delhi--110063, India}
\email{sushilkumar16n@gmail.com}

\author[Rakesh Kumar pandey]{Rakesh Kumar pandey}
\address{Department of Mathematics, University of Delhi, New Delhi-110063}
\email{rkpandey@maths.du.ac.in}

\author[Pratima Rai]{Pratima Rai}
\address{Department of Mathematical science university of Delhi, New Delhi}
\email{pratimarai5@gmail.com}

\textwidth=455pt \evensidemargin=8pt \oddsidemargin=8pt
\marginparsep=8pt \marginparpush=8pt \textheight=600pt
\topmargin=15pt
\parskip .3cm
\numberwithin{equation}{section}

\begin{abstract}
The goal of this manuscript to establish the best possible estimate on  coefficient functionals like Hermitian-Toeplitz determinant of secoend order involving  logarithmic coefficients,  initial logarithmic inverse coefficients and initial order Schwarzian derivatives of the Ozaki close-to-convex functions. 
\end{abstract}


\maketitle


\section{Introduction}
Let  $\mathcal{A}$ denote  the class  of  functions $f(z)=z+\sum_{n=2}^{\infty} a_n z^n$ which are analytic in  $\mathbb{D}= \vert z\vert <1$ and $\mathcal{S}$ be a subclass of $\mathcal{A}$ consists of  univalent functions. 
Let $\mathcal{P}$ denote the class of function $p(z)=1+\sum_{n=1}^{\infty}p_nz^n$ which are analytic and satisfy $\Re(p(z))>0$ for $\vert z\vert <1$. The notation of $h_1 \prec h_2$ means the analytic function $h_1$ is subordinate to  analytic function $h_2$ if there is a Schwarz function $w$ in $\vert z\vert <1$ such that $h_1(z) =h_2(w(z))$.  If $h_2\in \mathcal{S}$,  then $h_1\prec h_2$ if and only if $h_1(0)=h_2(0)$ and $h_1(|z|<1) \subset h_2(|z|<1)$. It means the behaviour of the function $h_1$ is subordinate or constrained by the function $h_2$ under some mapping~\cite{Duren}.  In 1941, Ozaki \cite{Ozaki} introduced the class $\mathcal{F}$   as 
$$\mathcal{F}=\bigg \{f\in\mathcal{A}: \Re\bigg(1+\frac{zf''(z)}{f'(z)}\bigg)>-\frac{1}{2}, \quad \mbox{for}\,\,\,\vert z\vert <1\bigg\}.$$
If the function $f\in\mathcal{F},$ then $f$ satisfies the  subordination relation $1+(zf''(z)/f'(z))\prec(1+2z)/(1-z),$
for $\vert z\vert <1$.  The function $f_1$ defined as $zf_1''(z)/f_1'(z)=3z/(1-z)$ or 
\begin{equation}\label{eqt2}
f_1(z)=z+\frac{3}{2}z^2+2z^3+\frac{5}{2}z^4+\cdots
\end{equation}
and the function $f_2$ defined as $zf_2''(z)/f_2'(z)=3z^2/(1-z^2)$ or
\begin{equation}\label{eqt1}
f_2(z)=z+\frac{1}{2}z^3+\frac{3}{8}z^5+\cdots
\end{equation}
belong to the class $\mathcal{F}.$ 
Further, author \cite{Ozaki}  studied the class $\mathcal{G}$ as
$$\mathcal{G}=\bigg \{f\in\mathcal{A}: \Re\bigg(1+\frac{zf''(z)}{f'(z)}\bigg)<\frac{3}{2},\quad \mbox{for}\,\,\,\vert z\vert <1\bigg\}.$$
If the function $f\in\mathcal{G},$ then $f$ satisfies the subordination relation
$1+(zf''(z)/f'(z))\prec (1-2z)/(1-z),$ $z\in\mathbb{D}.$   The function $g_1$ defined as $zg_1''(z)/g_1'(z))=(1-2z)/(1-z)$ or
\begin{equation}\label{eqt3}
g_1(z)=z-\frac{1}{2}z^2+\cdots
\end{equation}
and the function $g_2$ defined as $zg_2''(z)/g_2'(z))=(1-2z^2)/(1-z^2)$ or
\begin{equation}\label{eqt4}
g_2(z)=z-\frac{1}{6}z^3-\frac{1}{40}z^5+\cdots
\end{equation}
belong to the class $\mathcal{G}.$ Thus, the classes $\mathcal{F}$ and $\mathcal{G}$ are non-empty as well as the subclasses of close-to-convex functions. The functions in the classes $\mathcal{F}$ and $\mathcal{G}$  are known as \textbf{Ozaki close-to-convex functions} which have nice geometric properties and are used to understand the shape and behavior of various subclasses of univalent functions.

In the literature, bounds on coefficient functionals like Hermitian-Toeplitz determinants,  logarithmic coefficients,  logarithmic inverse coefficients and  Schwarzian derivatives have many important consequences in univalent function theory including the Koebe distortion theorem which describes how conformal maps distort shapes and it is named after Paul Koebe, a German mathematician who first proved the theorem in 1907, the Bieberbach conjecture which was proposed by the German mathematician Ludwig Bieberbach in 1916, and the other coefficient inequalities. The coefficient inequality states that for any function $f\in \mathcal{S}$, the absolute value of the coefficient $a_n$ satisfies $|a_n|\leq n$ for all $n\geq 2$. In 1978, Leung~\cite{leung78} discussed about successive coefficients for the class of starlike functions. Further, Li and Sugawa~\cite{Li17} studied successive coefficients of convex functions. Recently, authors~\cite{arora23} computed the bound on successive coefficients $||a_{n+1}|-|a_n||$ for the class of spirallike functions of non-negative functions. 
If the function $f\in\mathcal{S},$ then the function $F$ which is  inverse of the function $f$ has expansion
$F(w)=f^{-1}(w)=w+A_2w^2+A_3w^3+...$
or equivalently $f(F(w))=f(f^{-1}(w))=w\quad(|w|<r_0(f), r_0(f)\geq\frac{1}{4}).$
Furthermore, the Taylor series expansion  of the function $F$ is given as
$F(w)=w-a_2w^2+(2a_2^2-a_3)w^3-(5a_2^3-5a_2a_3+a_4)w^4+(14a_2^4-21a_2^2a_3+3a_3^2+6a_2a_4-a_5)w^5+...\cdot$
 Therefore, the initial inverse coefficients are given as
\begin{align}\label{eq13}
A_{2}=-a_2, \,\,\,A_{3}=-a_3+2a_2^2\,\,\,\mbox{and}\, \,\,A_{4}=-a_4+5a_2a_3-5a_2^3.
\end{align}
For details, one may refer~\cite{Ali03,Kapoor2007}. 
For the function $f\in\mathcal{S},$ 
the \textbf{ logarithmic coefficients} defined  as
$log ({f(z)}/{z})=2\sum_{n=1}^{\infty}\gamma_nz^n,$ $z\in\mathbb{D}$
 such that the initial logarithmic coefficients becomes
\begin{align}
\gamma_1=\frac{1}{2}a_2 
\quad\mbox{and}\quad
\gamma_2=\frac{1}{2}(a_3-\frac{1}{2}a_2^2) \label{eq2a}.
\end{align}
Alarifi~\cite{Najla22} calculated upper bound on third logarithmic coefficient for three subclasses of close-to-convex functions without any kind of restrictions. For more details related to logarithmic coefficients, one may refer~\cite{Alimoh2021,Girela2000}. 
Ponnusamy et.al.~\cite{Ponnusamy}  discussed the logarithmic coefficients $\Gamma_n,  n\in\mathbb{N}$ of inverse function $F$ as 
$log({F(w)}/{w})=2\sum_{n=1}^{\infty}\Gamma_nw^n,$ $  |w|<{1}/{4}$ 
where the coefficients $\Gamma_n$ are known as \textbf{logarithmic inverse coefficients}. Thus the  initial logarithmic inverse coefficients are given as
$\Gamma_1=\frac{1}{2}A_2,$ 
$\Gamma_2=\frac{1}{2}(A_3-\frac{1}{2}A_2^2),$ and $\Gamma_3=\frac{1}{2}(A_4-A_2A_3+\frac{1}{3}A_2^3)$. 
Further, on substituting the value of $A_2, A_3$ and $A_4$ from  \eqref{eq13}, we get
\begin{align}
\Gamma_1=&-\frac{1}{2}a_2 \label{eq43},\\
\Gamma_2=&-\frac{1}{2}(a_3-\frac{3}{2}a_2^2) \label{eq44},\\
\Gamma_3=&-\frac{1}{2}(a_4-4a_2a_3+\frac{10}{3}a_2^3)\label{eq45}.
\end{align}
For locally univalent functions $f$, the Schwarzian derivative   is defined by
$S_f(z)=\left(\frac{f''(z)}{f'(z)}\right)'-\frac{1}{2}\left(\frac{f''(z)}{f'(z)}\right)^2.$ Denote $\sigma_3(f)=S_f(z)$ and from \cite{esc}, the higher order Schwarzian derivative is given as
$	\sigma_{n+1}(f)=(\sigma_n(f))'-(n-1)\sigma_n(f) {f''}/{f'},$ $ n\geq 4.$
Without loss of generality, we assume that $\sigma_n(f)(0)=:\textbf{S}_n$ so that the third and fourth order Scharzian derivatives become
\begin{equation}\label{s34}
\textbf{S}_3=\sigma_3(f)(0)=6(a_3-a_2^2)\quad \mbox{and}\quad \textbf{S}_4=\sigma_4(f)(0)=24(a_4-3a_2a_3+2a_2^3).
\end{equation}
Nehari~\cite{ZNihari} gave a criteria of univalency of a analytic function using Schwarzian derivatives. 
For two natural numbers $q$ and $n$, the $q^\text{th}$ Hankel determinant $H_q(n)$  for the function $f\in \mathcal{S}$ is given by
$H_q(n):=\det \{a_{n+i+j-2}\}^q_{i,j}$, $ 1 \leq i,\,j \leq q, \, a_1=1$.
The $q^{th}$  Hermitian-Toeplitz determinant for the function $f\in \mathcal{S}$ is given by $T_{q,n}(F_{f}):=\det \{a_{ij}\}$, where $a_{ij}=a_{n+j-i}$ for $j\geq i$ and $a_{ij}=\overline{a_{ji}}$ for $j<i$.
Thus, the second order Hermitian-Toeplitz determinant for the function $f\in\mathcal{S}$ is given as
$T_{2,1}(F_{f})=\begin{vmatrix}
1 & a_{2} \\
\bar{a}_{2}& 1 
\end{vmatrix}=1-|a_2|^2.$
In terms of logarithmic coefficient, $T_{2,1}(F_{f}/2)$ becomes
$T_{2,1}(	F_{f}/\gamma)=\begin{vmatrix}
\gamma_{1} & \gamma_{2} \\
\bar{\gamma}_{2}& \gamma_1 
\end{vmatrix}=\gamma_1^2-|\gamma_2|^2.$
On substituting the values of $\gamma_1$ and $\gamma_{2}$ from    \eqref{eq2a},  we get
\begin{equation}\label{eq105}
T_{2,1}(F_{f}/\gamma)=\frac{1}{16}(-a_2^4+4a_2^2+4a_2^2\Re a_3-4|a_3|^2).
\end{equation}
We recall that in~\cite{Cudna}, authors computed sharp estimates on second and third order Hermitian-Toeplitz determinants involving initial coefficients of certain univalent functions. Obradovic and Tuneski~\cite{Tuneski}  computed the bounds on  Hermitian-Toeplitz determinant of third order involving initial coefficients of univalent functions. The authors~\cite{Rai22} computed   bounds on third order Hermitian-Toeplitz determinant for the  starlike tan hyperbolic functions.  The authors~\cite{mandal23} computed bounds on Hankel and Toeplitz determinants of Logarithmic coefficients of inverse functions for certain   univalent functions.
In~\cite{Allu19}, authors determined bounds on certain coefficients together with growth estimates for Ozaki close-to-convex functions. Further, the sharp bound on second Hankel determinant involving initial coefficients as well as inverse coefficients for the subclass of strongly Ozaki close-to-convex functions were determined in~\cite{sim21}. 
Further, authors~\cite{Tuneski22} improved upper bound of the third order Hankel determinant for the classes $\mathcal{F}$ and $\mathcal{G}$ respectively and they conjectured the sharpness of the bound. 
In a recent paper~\cite{Eker23}, authors established sharp bound on the second Hankel determinant involving logarithmic coefficients with invariance property of strongly Ozaki close-to-convex functions. 

Motivated by the above discussed literature, in the second section, we provide the sharp bounds on second order Hermitian-Toeplitz determinant $T_{2,1}(F_{f}/\gamma)$, initial logarithmic inverse coefficients $|\Gamma_i|$; $i=1,2,3$, third order Schwarzian derivative $|S_3|$ and the difference of successive inverse coefficients $|A_3-A_2|$ as well as the difference of logarithmic inverse coefficients $|\Gamma_3-\Gamma_2|$ respectively for the functions $f\in \mathcal{F}$. 
In the third section, we provide the sharp bounds on second order Hermitian-Toeplitz determinant $T_{2,1}(F_{f}/\gamma)$, initial logarithmic inverse coefficients $|\Gamma_i|$; $i=1,2,3$ and  third and fourth order Schwarzian derivatives $|S_3|$ and $|S_4|$ for the functions $f\in \mathcal{G}$. 

The following lemmas will play an important role in the  demonstration of  main results.
\begin{lem}\label{eqlem3}\cite{Sur1940}
Let $w(z)=c_1z+c_2z^2+c_3z^3+c_4z^4+...$ be a Schwarz function. Then 
\[|c_1|\leq1,\quad|c_2|\leq1-|c_1|^2,\quad|c_3|\leq1-|c_1|^2-\frac{|c_2|^2}{1+|c_1|}.\]
\end{lem}
\begin{lem}\label{lem1}\cite{Libera1983}
Let $\mathcal{P}$ be the class of analytic functions having the Taylor series of the form
\begin{equation}\label{pf}
p(z)=1+p_1z+p_2z^2+p_3z^3+\cdots
\end{equation}
satisfying the condition $\Re (p(z))>0\;(z\in\mathbb{D})$. Then
\begin{align*}
2p_2=&p_1^2+t\xi,\\
4p_3=&p_1^3+2p_1t\xi-p_1t\xi^2+2t(1-|\xi|^2)\eta,\\
8p_4=&p_1^4+3p_1^2t\xi+(4-3p_1^2)t\xi^2+p_1^2t\xi^3+4t(1-|\xi|^2)(1-|\eta|^2)\gamma\\
&\qquad\qquad\,\,+4t(1-|\xi|^2)(p_1\eta-p_1\xi\eta-\bar{\xi}\eta^2),
\end{align*}
for some $\xi ,\eta,\gamma\in \overline{\mathbb{D}}$ and  $t=(4-p_1^2).$
\end{lem}

\section{\textbf{The class $ \mathcal{F}$ }}
In this section, we first determine  sharp bounds on  $T_{2,1}(F_{f}/\gamma)$ for the functions $f\in \mathcal{F}$.

\begin{thm}
Let the function $f(z)=z+\sum_{n=2}^{\infty}a_nz^n$ be in the class $ \mathcal{F}$. Then  
$$-\frac{1}{16}\leq T_{2,1}(F_{f}/\gamma)\leq\frac{95}{256}.$$
The upper and lower bound are sharp for the functions $f_1$ and $f_2$ given by \eqref{eqt2} and \eqref{eqt1} respectively. 
\end{thm}
\begin{proof}
Let the function  $f\in\mathcal{F}.$  Then
$1+\frac{zf''(z)}{f'(z)}=\frac{1+2w(z)}{1-w(z)}$
 where $w(z)$ is a Schwarz function defined in  $\vert z\vert <1$. Since $p(z)={(1+w(z))}/{(1-w(z))}\in \mathcal{P},$  we have
\begin{equation}\label{eq201}
1+\frac{zf''(z)}{f'(z)}=\frac{3p(z)-1}{2}.
\end{equation}
The Taylors series expansions  of left hand side and right hand side   in expression  \eqref{eq201} are given as
\begin{align}\label{eq400}
1+\frac{zf''(z)}{f'(z)}=1&+2a_2z+(-4a_2^2+6a_3)z^2+(8a_2^3-18a_2a_3+12a_4)z^3\notag\\&+(-16a_2^4+48a_2^2a_3-18a_3^2-32a_2a_4+20a_5)z^4+\cdots
\end{align} and
\begin{align}\label{eq404}
\frac{3p(z)-1}{2}=1+\frac{3p_1z}{2} +\frac{3p_2z^2}{2}+\frac{3p_3z^3}{2}+\cdots
\end{align}
respectively. On comparing  \eqref{eq400} and \eqref{eq404}, we get initial coefficients 
\begin{align}
a_2=&\frac{3}{4}p_1\label{eq1},\\
a_3=&\frac{1}{8}(3p_1^2+2p_2)\label{eq1a},\\
a_4=&\frac{1}{64}(9p_1^3+18p_1p_2+8p_3)\label{eq1b}.
\end{align}
On substituting  the values of $a_2$ and $a_3$ in \eqref{eq105}, we get
\begin{align*}
T_{2,1}(	F_{f}/\gamma)=\frac{1}{16}\bigg(-\frac{81p_1^4}{256}+\frac{36p_1^2}{16}+\frac{36p_1^2}{128}\Re(3p_1^2+2p_2)-\frac{1}{16}|3p_1^2+2p_2|^2\bigg).
\end{align*}
By using  Lemma \ref{lem1}, the  above expression becomes
\begin{equation}\label{eq4}
T_{2,1}(	F_{f}/\gamma))=\frac{1}{4096}(-49p_1^4+576p_1^2-56p_1^2(4-p_1^2)\Re(\xi)-16(4-p_1^2)^2|\xi|^2).
\end{equation}
Using  $-\Re(\xi)\leq|\xi|,$ and  setting $x=|\xi|\in[0,1]$ and $p_1=p,$ expression   \eqref{eq4} becomes
\begin{align*}
T_{2,1}(	F_{f}/\gamma)&\leq\frac{1}{4096}(-49p_1^4+576p_1^2+56p_1^2(4-p_1^2)x-16(4-p_1^2)^2x^2)=\Upsilon(p,x).
\end{align*}
In order to prove our result,  we determine  the maximum value of the function $\Upsilon$ over  rectangular region $\Omega=[0,2]\times[0,1].$ We consider two cases on region  $\Omega$ for the same.
\begin{itemize}
\item[(1)] First, the boundary points of $\Omega$ are being  considered. A simple calculation gives
\begin{align*}
 &\Upsilon(0,x)=-\frac{x^2}{16}\leq0,\quad\Upsilon(2,x)=\frac{95}{256},\quad \Upsilon(p,0)=\frac{1}{4096}(-49p^4+576p^2)\leq\frac{9   5}{256},\\&\mbox{and}\quad\Upsilon(p,1)=\frac{1}{4096}(-121p^4+928p^2-256)\leq\frac{95}{256}.
\end{align*}
\item[$(2)$] Next, the interior points of $\Omega$ are being considered. A solution of  system of equations ${\partial \Upsilon(p,x)}/{\partial p}=0$ and ${\partial \Upsilon(p,x)}/{\partial x}=0$ refers a critical point of $\Upsilon.$ The equation   ${\partial \Upsilon(p,x)}/{\partial x}=0$ gives $x=7p^2/(4(4-p^2))=x_p\in(0,1)$ which  holds for $p<{\sqrt{16/11}}\in(0,2).$ Further we solve ${\partial \Upsilon(p,x)}/{\partial p}=0$ and substituting the value $x_p$ we get $p^5-8p^3+16p=0$ which is never possible for $p\in(0,2).$
Thus, the function $\Upsilon$ has no maximum value in the interior of $\Omega.$
\end{itemize}
From the  cases (1) and (2), we  conclude the desire upper bound  ${95}/{256}$ on  $T_{2,1}(F_{f}/\gamma)$.

Using  inequality $-\Re(\xi)\geq-|\xi|$ and setting $x=|\xi|$  in  \eqref{eq4},  we get
\begin{align*}
T_{2,1}(F_{f}/\gamma)&\geq\frac{1}{4096}(-49p^4+576p^2-56p^2(4-p^2)x-16(4-p^2)^2x^2)\notag\\&=\Psi(p,x).
\end{align*}
On the  boundary of $\Omega$, a simple calculation on $\Psi(p,x)$ yields
\begin{align*}
 &\Psi(0,x)=-\frac{x^2}{16}\geq-\frac{1}{16},\quad\Psi(2,x)=\frac{95}{256},\quad\Psi(p,0)=\frac{1}{4096}(-49p^4+576p^2)\geq0,\\
 &\mbox{and}\,\,\,\Psi(p,1)=\frac{1}{4096}(-9p^4+480p^2-256)\geq-\frac{1}{16}.
\end{align*} 
In the interior of $\Omega$,   we note that ${\partial \Psi(p,x)}/{\partial x}={-((4-p^2)(56p^2+32x(4-p^2)))}/4096\neq0$.  The function $\Psi$ has no minimum value in the interior region of $\Omega.$  
Therefore, we  conclude the desire lower bound  $-1/16$ on  $T_{2,1}(F_{f}/\gamma)$.\qedhere
\end{proof}
In the next result, we compute the sharp bounds on  $|\Gamma_i|$; $i=1,2,3$ for the functions $f\in \mathcal{F}$. 
\begin{thm}\label{thm1}
Let the function $f(z)=z+\sum_{n=2}^{\infty}a_nz^n$ be in the class $ \mathcal{F}$. Then 
$$|\Gamma_1|\leq\frac{3}{4},\quad|\Gamma_2|\leq\frac{11}{16},\quad|\Gamma_3|\leq\frac{7}{8}.$$
All inequalities are sharp for the function $f_1$ given in \eqref{eqt2}.
\end{thm}
\begin{proof}
As before, if the function $f\in\mathcal{F},$  then there exist a Schwarz function $w(z)=\sum_{{k=1}}^{\infty}c_kz^k$ in  $\vert z\vert <1$ such that 
$1+\frac{zf''(z)}{f'(z)}=\frac{1+2w(z)}{1-w(z)}.$
A simple calculation gives
\begin{align}\label{eq401}
\frac{1+2w(z)}{1-w(z)}=1&+3c_1z+3(c_1^2+c_2)z^2+3(c_1^3+2c_1c_2+c_3)z^3\notag \\
&+3(c_1^4+3c_1^2c_2+c_2^2+2c_1c_3+c_4)z^4+\cdots
\end{align}
On comparing \eqref{eq400} and \eqref{eq401},  the initial coefficients in terms of Schwarz function become
\begin{align}
a_2=&\frac{3}{2}c_1\label{eq1f},\\
a_3=&\frac{1}{2}(4c_1^2+c_2)\label{eq1g},\\
a_4=&\frac{1}{8}(20_1^3+13c_1c_2+2c_3)\label{eq1h}.
\end{align}
Using Lemma \eqref{eqlem3}, from \eqref{eq43} and \eqref{eq1f}, we get  $|\Gamma_1|=3|c_1|/4\leq3/4$.
Further, using \eqref{eq1f} and \eqref{eq1g} in \eqref{eq44}, we get
\begin{align*}
|\Gamma_2|&=\frac{1}{16}|-11c_1^2+4c_2|.
\end{align*}
Using Lemma \ref{eqlem3}, we get
\begin{align*}
|\Gamma_2|&\leq\frac{1}{16}(7|c_1|^2+4)\leq\frac{11}{16}.
\end{align*}
In view of \eqref{eq1f}, \eqref{eq1g},\eqref{eq1h} and \eqref{eq45}, we have
\begin{align*}
|\Gamma_3|&=\frac{1}{48}\bigg|42c_1^3-33c_1c_2+6c_3\bigg|\notag\\&\leq\frac{1}{48}(42|c_1|^3+33|c_1||c_2|+6|c_3|).
\end{align*}
By making use of Lemma \ref{eqlem3}, we get
\begin{align*}
|\Gamma_3|&\leq\frac{1}{48}\bigg(42|c_1|^3+33|c_1||c_2|+6(1-|c_1|^2-\frac{|c_2|^2}{1+|c_1|})\bigg)=\chi(|c_1|,|c_2|).
\end{align*}
Next, we find the maximum value of the function $\chi$ over the region  $\Lambda=\{(|c_1|,|c_2|): |c_1|\leq1, |c_2|\leq1-|c_1|^2 \}.$
The equation ${\partial \chi}/{\partial |c_2|}=0$ provides $|c_2|=(33|c_1|(1+|c_1|))/12\in(0,1)$   when $|c_1|\in (0,\frac{-33+\sqrt{1221}}{66}).$ 
On substituting the value of $|c_2|$ in equation 
$$\frac{\partial \chi}{\partial |c_1|}=\frac{1}{48}\bigg(126|c_1|^2-12|c_1|+33|c_2|+\frac{6|c_2|^2}{(1+|c_1|)^2}\bigg)=0,$$
we get  $6291|c_1|^2+1890|c_1|=0,$ which is not possible. Therefore, it is noted that the function $\chi$ has no  maximum value in the interior of $\Lambda.$
The continuity of the function $\chi$ over  the compact region $\Lambda$ ensure  the  maximum value of $\chi$ attains at boundary of $\Lambda.$ Therefore, we have
\begin{enumerate}
\item [(i)]$\chi(0,|c_2|)={(1-|c_2|^2)}/{8}\leq1/8,$  for all $0\leq |c_2|\leq1,$
\item[(ii)]$\chi(|c_1|,0)=({42|c_1|^3-6|c_1|^2+6})/{48}\leq{7}/{8}$ for all $0\leq |c_1|\leq1,$
\item[(iii)]$\chi(|c_1|,1-|c_1|^2)=({3|c_1|^3+39|c_1|})/48\leq{7}/{8}$ for all $0\leq |c_1|\leq1.$
\end{enumerate}
Thus, the maximum value of $\chi$ is $7/8$ in domain $\Lambda.$
\qedhere
\end{proof}

Next result provides the sharp estimate on  $|S_3|$ for the functions $f\in \mathcal{F}$. 
\begin{thm}
Let the function $f(z)=z+\sum_{n=2}^{\infty}a_nz^n$ be in the class $ \mathcal{F}$. Then  
\[|S_3|\leq3.\]
The inequality is  sharp for the  function $f_2$ given by \eqref{eqt1}.
\end{thm}
\begin{proof}
In view of \eqref{s34}, \eqref{eq1f} and  \eqref{eq1g}, we have
\[|S_3|=\frac{3}{2}|-c_1^2+2c_2|\leq\frac{3}{2}(|c_1|^2+2|c_2|).\]
By Lemma \ref{eqlem3}, for $0\leq|c_1|\leq1$, we conclude the result as 
\[|S_3|\leq\frac{3}{2}(-|c_1|^2+2)\leq3.\qedhere\]
\end{proof}
Next, we compute the sharp bound on   $|A_3-A_2|$ and  $|\Gamma_3-\Gamma_2|$ for the functions $f\in \mathcal{F}$. 
\begin{thm}
Let the function $f(z)=z+\sum_{n=2}^{\infty}a_nz^n$ be in the class $ \mathcal{F}$. Then  
\begin{itemize}
\item [(a)] $|A_3-A_2|\leq4,$
\item [(b)] $|\Gamma_3-\Gamma_2|\leq\dfrac{25}{16}.$
\end{itemize}
Both the inequalities are sharp for the function given by \eqref{eqt2}.
\end{thm}
\begin{proof}
(a) Using   \eqref{eq13}, \eqref{eq1f} and \eqref{eq1g}, we have
\[|A_3-A_2|=\frac{1}{2}|5c_1^2+3c_1-c_2|.\]
On applying Lemma \eqref{eqlem3}, we get
\begin{align*}
|A_3-A_2|\leq\frac{1}{2}(4|c_1|^2+3|c_1|+1)\leq4, \quad |c_1|\leq1.
\end{align*}
(b)  In view of \eqref{eq44}, \eqref{eq45}, \eqref{eq1f}, \eqref{eq1g} and \eqref{eq1h}, we have
$$|\Gamma_3-\Gamma_2|=\frac{1}{48}\bigg|(-42c_1^3+33c_1c_2-33c_1^2+12c_2-6c_3)\bigg|.$$
Using Lemma \eqref{eqlem3}, we get
\begin{align*}
|\Gamma_3-\Gamma_2|&\leq\frac{1}{48}\bigg(42|c_1|^3+33|c_1| |c_2|+33|c_1|^2+12|c_2|+6(1-|c_1|^2-\frac{|c_2|^2}{1+|c_1|})\bigg)\notag\\&=M(|c_1|,|c_2|).
\end{align*}
Next, we find the maximum value of the function $M$ over the region  $\Lambda=\{(|c_1|,|c_2|): |c_1|\leq1, |c_2|\leq1-|c_1|^2 \}.$ First, let us consider  the boundary of $\Lambda$. We have
\begin{enumerate}
\item [(i)]$M(0,|c_2|)={(-|c_2|^2+2|c_2|+1)}/{8}\leq1/4,$\quad $0\leq |c_2|\leq1,$
\item[(ii)]$M(|c_1|,0)=({42|c_1|^3+27|c_1|^2+6})/{48}\leq{25}/{16},$\quad $0\leq |c_1|\leq1,$
\item[(iii)]$M(|c_1|,1-|c_1|^2)=({3|c_1|^3+21|c_1|^2+39|c_1|+12})/48\leq{25}/{16}$,\quad $0\leq |c_1|\leq1.$
\end{enumerate}
It is noted  that ${\partial M}/{\partial |c_2|}=\bigg(33|c_1|+12(1-{|c_2|}/{(1+|c_1|)})\bigg)/48\neq0$ in the interior of $\Lambda$. Thus, the function $M$ has no critical point in the interior of $\Lambda$. Thus, we conclude the result.\qedhere
\end{proof}

\section{\textbf{The class $ \mathcal{G}$} }
In this section,  we first do investigation about the  sharpness of the  bounds on second order Hermitian-Toeplitz determinant $T_{2,1}(F_{f}/\gamma)$ for the functions $f\in \mathcal{G}$. 
\begin{thm}
Let $f(z)=z+\sum_{n=2}^{\infty}a_nz^n$ be in  the class  $\mathcal{G}.$ Then 
$$-\frac{1}{144}\leq T_{2,1}(	F_{f}/\gamma)\leq\frac{15}{256}.$$
The upper and lower bound are sharp for the function $g_1$ and $g_2$  given  by \eqref{eqt3} and \eqref{eqt4} respectively.
\end{thm}
\begin{proof}
Let the function $f\in\mathcal{G}.$  Then
$1+\frac{zf''(z)}{f'(z)}=\frac{1-2w(z)}{1-w(z)}$
where $w(z)$ is a Schwarz function defined on $\mathbb{D}.$ Since
$p(z)={(1+w(z))}/{(1-w(z))}\in\mathcal{P},$ then we have
\begin{equation}\label{eq201a}
1+\frac{zf''(z)}{f'(z)}=\frac{3-p(z)}{2}.
\end{equation}
The  series expansion of right hand side of   \eqref{eq201a} becomes
\begin{align}\label{eq407}
\frac{3-p(z)}{2}=1-\frac{p_1z}{2} -\frac{p_2z^2}{2}-\frac{p_3z^3}{2}-\frac{p_4z^4}{2}\cdots
\end{align}
On comparing \eqref{eq400} and \eqref{eq407}, we get
\begin{align}
a_2=&-\frac{p_1}{4}\label{eq1j},\\
a_3=&\frac{1}{24}(p_1^2-2p_2)\label{eq1k},\\
a_4=&\frac{1}{192}(-p_1^3+6p_1p_2-8p_3)\label{eq1m}.
\end{align}
On putting the values of $a_2$ and $a_3$  from  \eqref{eq1j} and \eqref{eq1k} in \eqref{eq105}, we have
\begin{align*}
T_{2,1}(	F_{f}/\gamma)=\frac{1}{16}\bigg(-\frac{p_1^4}{256}+\frac{p_1^2}{4}+\frac{p_1^2}{96}\Re(p_1^2-2p_2)-\frac{1}{144}|p_1^2-2p_2|^2\bigg).
\end{align*}
Using Lemma \ref{lem1}, above expression becomes 
\begin{equation}\label{eq4a}
T_{2,1}(	F_{f}/\gamma))=\frac{1}{36864}(-9p_1^4+576p_1^2-24p_1^2(4-p_1^2)\Re(\xi)-16(4-p_1^2)^2|\xi|^2).
\end{equation}
Using  $-\Re(\xi)\leq|\xi|,$ and  setting $x=|\xi|\in[0,1]$ (with $p_1=p$), we get 
\begin{align*}
T_{2,1}(	F_{f}/\gamma)&\leq\frac{1}{36864}(-9p_1^4+576p_1^2+24p_1^2(4-p_1^2)x-16(4-p_1^2)^2x^2)=\Phi(p,x).
\end{align*}
In order to prove our result, we determine the maximum value of the function $\Phi$ over  rectangular region $\Omega=[0,2]\times[0,1].$ 
\begin{itemize}
\item[(1)] At the boundary points of $\Omega,$ It is noted that 
\begin{align*}
 &\Phi(0,x)=-\frac{x^2}{144}\leq0,\quad\Phi(2,x)=\frac{15}{256},\quad \Phi(p,0)=\frac{1}{36864}(-9p^4+576p^2)\leq\frac{1   5}{256},\\&\mbox{and} \,\,\Phi(p,1)=\frac{1}{36864}(-49p^4+800p^2-256)\leq\frac{15}{256}.
\end{align*}
\item[$(2)$]To compute the critical points of the function  $\Phi,$ the system of equations  ${\partial \Phi}/{\partial p}=0$ and ${\partial \Phi}/{\partial x}=0$ must have a solution in the interior of $\Omega.$ The equation   ${\partial \Phi}/{\partial x}=0$ gives $x=3p^2/(4(4-p^2))=x_p\in(0,1)$ which holds for $p<{\sqrt{16/7}}\in(0,2).$ Next, on putting the value of $x_p$ in equation ${\partial \Phi}/{\partial p}=0$, we obtain the equation $p^5-8p^3+16p=0$ which is not possible for any $p\in(0,2)$.
Therefore, the function $\Phi$ has no maximum value in the interior of $\Omega.$ 
\end{itemize}
Hence, from  cases (1) and (2), we  conclude that the best upper bound on $T_{2,1}(	F_{f}/\gamma)$ is ${15}/{256}.$
Next,  on using  $-\Re(\xi)\geq-|\xi|$ and setting $x=|\xi|$  in  \eqref{eq4a},  we get
\begin{align*}
T_{2,1}(F_{f}/\gamma)&\geq\frac{1}{36864}(-9p^4+576p^2-24p^2(4-p^2)x-16(4-p^2)^2x^2)=N(p,x).
\end{align*}
Next, we determine the minimum value of the function $N(p,x)$ over  rectangular region $\Omega=[0,2]\times[0,1].$ 
First, we consider boundary points  of $\Omega$.
\begin{align*}
 &N(0,x)=-\frac{x^2}{144}\geq-\frac{1}{144},\quad N(2,x)=\frac{15}{256},\quad N(p,0)=\frac{1}{36864}(-9p^4+576p^2)\geq0,\\&\mbox{and}\,\,N(p,1)=\frac{1}{36864}(-p^4+608p^2-256)\geq-\frac{1}{144}.
\end{align*}
Since   ${\partial N}/{\partial x}={-((4-p^2)(24p^2+32x(4-p^2)))}/36864\neq0$ for all $(p,x)\in(0,2)\times(0,1).$  Thus, the function $N$ has no minimum value in the interior of $\Omega.$ 
Therefore,  we conclude the minimum value of function $N$ over $\Omega$ is $-1/144.$\qedhere 
\end{proof}
Next, we compute sharp bound on initial logarithmic inverse coefficients $|\Gamma_i|$; $i=1,2,3$ for the functions $f\in \mathcal{G}$. 
\begin{thm} Let $f(z)=z+\sum_{n=2}^{\infty}a_nz^n$ be in the class  $\mathcal{G}.$ Then 
$$|\Gamma_1|\leq\frac{1}{4},\quad|\Gamma_2|\leq\frac{3}{16},\quad|\Gamma_3|\leq\frac{5}{24}.$$
All three inequalities  are sharp  for the function $g_1$ given by \eqref{eqt3}.
\end{thm}
\begin{proof}
Since the function  $f\in\mathcal{G},$  then there exist a Schwarz function $w(z)=\sum_{{k=1}}^{\infty}c_kz^k$ in  $\vert z\vert <1$ such that
$1+\frac{zf''(z)}{f'(z)}=\frac{1-2w(z)}{1-w(z)}.$
A simple calculation yields
\begin{align}\label{eq405}
\frac{1-2w(z)}{1-w(z)}=1&+3c_1z+3(c_1^2+c_2)z^2+3(c_1^3+2c_1c_2+c_3)z^3\notag\\&+3(c_1^4+3c_1^2c_2+c_2^2+2c_1c_3+c_4)z^4+\cdots
\end{align}
On comparing  \eqref{eq400} and \eqref{eq405}, the initial coefficients become
\begin{align}
a_2=&-\frac{c_1}{2}\label{eq1u},\\
a_3=&-\frac{c_2}{6}\label{eq1v},\\
a_4=&-\frac{1}{24}(c_1c_2+2c_3)\label{eq1w}.
\end{align}
In view of \eqref{eq43} and \eqref{eq1u}, we have $|\Gamma_1|=|c_1|/4\leq1/4$. Further, from 
  \eqref{eq1u}, \eqref{eq1v} and  \eqref{eq44}, we have
\begin{align*}
|\Gamma_2|&=\frac{1}{48}|9c_1^2+4c_2|.
\end{align*}
By Lemma \eqref{eqlem3}, we get
$
|\Gamma_2|\leq(5|c_1|^2+4)/48\leq {3}/{16}.
$
Next, in view of \eqref{eq1u}, \eqref{eq1v},\eqref{eq1w} and \eqref{eq45}, we have
\begin{align*}
|\Gamma_3|&=\frac{1}{48}\bigg|10c_1^3+9c_1c_2+2c_3\bigg|\notag\\&\leq\frac{1}{48}(10|c_1|^3+9|c_1||c_2|+2|c_3|).
\end{align*}
Using Lemma \eqref{eqlem3}, we get
\begin{align*}
|\Gamma_3|&\leq\frac{1}{48}\bigg(10|c_1|^3+9|c_1||c_2|+2(1-|c_1|^2-\frac{|c_2|^2}{1+|c_1|})\bigg)\notag=S(|c_1|,|c_2|).
\end{align*}
To determine the maximum value of the function $S$ over the region  $\Lambda=\{(|c_1|,|c_2|): |c_1|\leq1, |c_2|\leq1-|c_1|^2 \},$  we  consider following cases
\begin{itemize}
\item [$(1)$] On  the boundary of $\Lambda$, we have
\begin{enumerate}
\item [(i)]$S(0,|c_2|)={(1-|c_2|^2)}/{24}\leq1/24,$  for all $0\leq |c_2|\leq1,$
\item[(ii)]$S(|c_1|,0)=({10|c_1|^3-2|c_1|^2+2})/{48}\leq{5}/{24}$ for all $0\leq |c_1|\leq1,$
\item[(iii)]$S(|c_1|,1-|c_1|^2)=({-|c_1|^3+11|c_1|})/48\leq{5}/{24}$ for all $0\leq |c_1|\leq1.$
\end{enumerate}
\item [$(2)$] We take the interior region of  $\Lambda$. A simple calculation gives 
 $$\frac{\partial S}{\partial |c_1|}=\frac{1}{48}\bigg(30|c_1|^2-4|c_1|+9|c_2|+\frac{2|c_2|^2}{(1+|c_1|)^2}\bigg)=0,$$ and $$\frac{\partial S}{\partial |c_2|}=\frac{1}{48}\bigg(9|c_1|-\frac{4|c_2|}{1+|c_1|}\bigg)=0.\quad$$
Using similar lines done in  Theorem \ref{thm1}, we conclude that the system of equations ${\partial S}/{\partial |c_1|}=0$ and ${\partial S}/{\partial |c_2|}=0$  has no common solution in the interior of $\Lambda$. Hence, the function $S$ has no maximum value in the interior of $\Lambda.$ 
\end{itemize}
From the cases $(1)$ and $(2)$, we obtain desired estimate  on $|\Gamma_3|$.\qedhere
\end{proof}

Next result provides the sharp bounds on third and fourth order Schwarzian derivatives $|S_3|$ and $|S_4|$ respectively for the functions $f\in \mathcal{G}$. 
\begin{thm}
Let the function $f(z)=z+\sum_{n=2}^{\infty}a_nz^n$ be in  the class  $\mathcal{G}.$ Then
\[|S_3|\leq\frac{3}{2}\,\,\,\mbox{and}\,\,\, |S_4|\leq6.\]
The both the inequalities are sharp for the function given by \eqref{eqt3}.
\end{thm}
\begin{proof}
(a) In view of \eqref{eq1u}, \eqref{eq1v} and \eqref{eq1g}, we have
\[|S_3|=\frac{1}{2}|3c_1^2+2c_2|\leq\frac{1}{2}(3|c_1|^2+2|c_2|).\]
Using Lemma \ref{eqlem3}, above expression becomes
\[|S_3|\leq\frac{1}{2}(|c_1|^2+2)\leq\frac{3}{2},\quad0\leq|c_1|\leq1.\]

(b) In view of \eqref{eq1u}, \eqref{eq1v} and \eqref{eq1g}, we get 
\[|S_4|=|6c_1^3+7c_1c_2+2c_3|\leq(6|c_1|^3+7|c_1| |c_2|+2|c_3|).\]
From Lemma \ref{eqlem3}, the above expression becomes
\begin{align*}
|S_4|&\leq6|c_1|^3+7|c_1| |c_2|+2(1-|c_1|^2-\frac{|c_2|^2}{1+|c_1|})\notag=\delta(|c_1|,|c_2|).
\end{align*}
Next, we find the maximum value of the function $\delta$ over the region  $\Lambda=\{(|c_1|,|c_2|): |c_1|\leq1,|c_2|\leq1-|c_1|^2 \}.$ 
On the boundary of $\Lambda$, it is noted that
\begin{enumerate}
\item [(i)]$\delta(0,|c_2|)={2(1-|c_2|^2)}\leq2,$  for all $0\leq |c_2|\leq1,$
\item[(ii)]$\delta(|c_1|,0)=({6|c_1|^3-2|c_1|^2+2})\leq6$ for all $0\leq |c_1|\leq1,$
\item[(iii)]$\delta(|c_1|,1-|c_1|^2)=({-3|c_1|^3+9|c_1|})\leq6$ for all $0\leq |c_1|\leq1.$
\end{enumerate}
Next, we consider the interior of $\Lambda$.  The equation  ${\partial \delta}/{\partial |c_2|}=0$  gives $|c_2|=(7|c_1|(1+|c_1|))/4\in(0,1)$ which is true for  $|c_1|\in (0,\frac{-7+\sqrt{161}}{14}).$ Further, we solve $$\frac{\partial \delta}{\partial |c_1|}=\bigg(18|c_1|^2-4|c_1|+7|c_2|+\frac{2|c_2|^2}{(1+|c_1|)^2}\bigg)=0,$$ and substituting the values of $|c_2|$ in the previous equation  which gives $291|c_1|^2+66|c_1|=0,$ which is not possible. Therefore, the function $\delta$ has no maximum value in the interior of $\Lambda.$\qedhere
\end{proof}
\section*{Acknowledgement}
The first and the third author express their thanks to the Institute of Eminence, University of Delhi, Delhi, India--110007 for providing financial support for this research under grant number-Ref. No./IoE/2023-24/12/FRP. The second author would like to thank to  the UGC Non-NET Fellowship for supporting financially vide Ref. No. Sch/139/Non-NET/Ext-156/2022-2023/722.

\end{document}